\documentclass{amsart}
\usepackage{amsmath,amsthm,amssymb,IMjournal}
\usepackage{pdfsync}
\usepackage{latexsym}
\usepackage{amsgen,amsmath,amsxtra,amssymb,amsfonts,amsthm}
\usepackage{times}
\usepackage{graphicx}% Include figure files
\usepackage[dvips]{epsfig}% Include figure files
\usepackage{subfigure}
\usepackage[active]{srcltx}
\usepackage{color}

\begin{document}
\newtheorem{The}{Theorem}[section]
\newtheorem{Con}{Conjecture}[section]
\newtheorem{proposition}{Proposition}[section]

\numberwithin{equation}{section}

\title{Existence results for a fourth order partial differential equation arising in condensed matter physics}

\author{\|Carlos |Escudero|, Madrid,
        \|Filippo |Gazzola|, Milan,
        \|Robert |Hakl|, Brno,
        \|Ireneo |Peral|, Madrid,
        \|Pedro J. |Torres|, Granada}

\rec {September 17, 2013}

\abstract
%%%%%%
%%%7 to 12 lines optimum; references in the abstract should be given in full form, e.g. V. Nov\'ak, M. Novotn\'y: Linear extensions of orderings. Czech. Math. J. 50 (2000), 853--864.
%%%%%
   We study a higher order parabolic partial differential equation that arises in the context of condensed matter physics. It is a
   fourth order semilinear equation whose nonlinearity is the determinant of the Hessian matrix of the solution. We consider this model in
   a bounded domain of the real plane and study its stationary solutions both when the geometry of this domain is arbitrary and when it is
   the unit ball and the solution is radially symmetric. We also consider the initial-boundary value problem for the full parabolic equation.
   We summarize our results on existence of solutions in these cases and propose an open problem related to the existence of self-similar
   solutions.
\endabstract



\thanks
   The research has been supported by projects MTM2010-18128, RYC-2011-09025 and SEV-2011-0087, MINECO, Spain,
   and project RVO 67985840, Czech Republic.
\endthanks

\section{Introduction}\label{int}

We are interested in the initial-boundary problem for $u=u(x,y,t)$ solving the following parabolic equation
\begin{equation}\label{parabolicpde}
\partial_t u + \Delta^2 u = \det(D^2 u) + \lambda h,
\end{equation}
subject to the initial condition $u(x,y,0)=u_0(x,y)$ and where $h$ is some function depending in general on both space and time coordinates
and belonging to some suitable Lebesgue space, $\lambda \in \mathbb{R}$.
This equation is to be solved for $(x,y) \in \Omega \subset \mathbb{R}^2$, where $\Omega$ is open, bounded and provided with a smooth boundary,
and $t > 0$. We will consider two different sets of boundary conditions: Dirichlet and Navier.

This partial differential equation is a model that arises in the coarse-grained description of epitaxial growth
processes in the field of condensed matter physics. The function $u=u(x,y,t)$ models the height of the growing film
at the spatial point $(x,y)$ at time $t$. Note that the evolution is dictated by the competition among the determinant
of the Hessian matrix of the solution and the bilaplacian. These terms model the dynamics at the solid surface. The function $h$
models the introduction of new mass on the surface, and the parameter $\lambda$ measures the intensity of this income of new mass.
A formal derivation of this model in terms of geometric quantities can be found in~\cite{n2} and references therein.

We have partially analyzed this model in a series of recent works~\cite{n1,n2,n3,n4}.
We have shown the existence of stationary solutions to
this partial differential equation, in the case $h$ were time independent,
for two different sets of boundary conditions: Dirichlet and Navier. The proofs are different for the
two sets of boundary conditions because the variational structure that is present in the Dirichlet problem
is absent in the Navier one~\cite{n4}. For solutions that are radially symmetric we recover the variational structure
in both cases, and the proofs of existence of solutions are built making an explicit use of this fact~\cite{n2}.
Furthermore, in the case of radially symmetric solutions it is possible to prove non-existence of solutions
for large enough data~\cite{n3}. For the evolution problem it is possible to prove local existence of solutions
for arbitrary data and global existence of solutions for small (but otherwise arbitrary) data. Depending on the boundary conditions and
the concomitant presence of a variational structure in the equation as well as on the size of the data
it is possible to prove blow-up of the solution in finite time and convergence to a stationary solution
in the long time limit~\cite{n1}.

A summary of these results will be exposed in the next section. Together with these proven facts, there is a number of
questions that remain open for both partial differential equation~\eqref{parabolicpde} and its stationary counterpart. One of them is
the existence of self-similar solutions and its possible role in the blow-up structure. We describe an open question related to
the existence of self-similar solutions in section~\ref{sss}.

\section{Summary of previous results concerning existence of solutions}\label{summ}

In this section we list without proof some results we have recently obtained. The proofs can be found in the references quoted in the
text next to the statement in question. We are going to consider the two following sets of boundary conditions
$u= \partial_n u =0$ on $\Omega$,
which we refer to as Dirichlet boundary conditions, and
$u= \Delta u =0$ on $\Omega$,
which we refer to as Navier boundary conditions.
Following~\cite{n4} we find that the Dirichlet problem for the stationary version of~\eqref{parabolicpde}
\begin{equation}\label{ellipticpde}
\Delta^2 u = \det(D^2 u) + \lambda h,
\end{equation}
where of course $h$ is assumed to be time independent,
has a variational structure:

\begin{The}\label{T1}
The critical points of the functional
\begin{eqnarray}\nonumber
J_\lambda:W^{2,2}_0(\Omega)&\rightarrow& \mathbb{R} \\
\label{functional}
u &\rightarrow& J_\lambda(u)=\displaystyle \frac12\int_\Omega |\Delta u|^2 \, dx \, dy
-\int_\Omega u_{x} u_{y} u_{x y} \, dx \, dy - \lambda \int_\Omega h u \, dx \, dy,
\end{eqnarray}
are weak solutions to the Dirichlet problem for~\eqref{ellipticpde}.
\end{The}

This variational structure can be used to prove the existence of at least two solutions to the corresponding boundary value problem.

\begin{The}\label{T2}
Let $h\in L^1(\Omega)$. Then there
exists a $\lambda_0 >0$ such that for $0 \le \lambda<\lambda_0$, the Dirichet problem for
equation~\eqref{ellipticpde} has at least two solutions in $W^{2,2}_0(\Omega)$.
\end{The}

The proof can be found in~\cite{n4} and makes use of the mountain pass geometry of functional~\eqref{functional}.
For the Navier problem for equation~\eqref{ellipticpde} we cannot use the above mentioned variational methods because we do not know
of any suitable functional in this case. Instead we have the following result:

\begin{The}\label{T3}
Let $h\in L^1(\Omega)$. Then there
exists a $\lambda_0 >0$ such that for $0 \le \lambda<\lambda_0$, the Navier problem for
equation~\eqref{ellipticpde} has at least one solution in $W^{1,2}_0(\Omega) \cap W^{2,2}(\Omega)$.
\end{The}

The proof makes use of Banach fixed point theorem, see~\cite{n4}. Note also that this proof can be immediately adapted
for the case of Dirichlet boundary conditions.

The radial problem corresponding to~\eqref{ellipticpde} reads
\begin{equation}
\nonumber \frac{1}{r} \left\{ r \left[ \frac{1}{r} \left(
r u' \right)' \right]' \right\}' = \frac{1}{r} \,
u' u'' + \lambda h(r),
\end{equation}
where $r=\sqrt{x^2+y^2}$. In this case, for both sets of boundary conditions, the problem admits a variational formulation.
And in both cases, due to the mountain pass geometry of the associated functional, it is possible to prove the existence of
at least two solutions for small enough $\lambda$, see~\cite{n2}. Furthermore, it is also possible to prove the non-existence
of solutions for large enough $\lambda$ as well as rigorous bounds for the values of $\lambda$ that separate existence from non-existence,
see~\cite{n3}. Note that these bounds are rather precise in certain cases when compared to the numerical estimations of the critical values
of $\lambda$ calculated in~\cite{n2}.
With respect to the full evolution problem~\eqref{parabolicpde}, we can prove the following theorem:

\begin{The}
Let $T>0$; for any $u_0\in W_0^{2,2}(\Omega)$, any $h\in L^2(0,T;L^2(\Omega))$ and any $\lambda \in \mathbb{R}$
the Dirichlet problem has a unique
solution in
$$
\mathcal{X}_T:=C(0,T;W_0^{2,2}(\Omega)) \cap L^2(0,T;W^{4,2}(\Omega)) \cap W^{1,2}(0,T;L^2(\Omega)),
$$
provided $T$ is sufficiently small.
Furthermore, for any $T \in (0,\infty]$ there exists a unique solution to this problem in the same space provided $\|u_0\|_{W^{2,2}}$
and $\lambda$ are small enough.
Moreover, if $[0,T^*)$ denotes the maximal interval of continuation of $u$ and if $T^*<\infty$ then $\|u(t)\|_{W^{2,2}}\to \infty$ as $t\to T^*$.
\end{The}

An analogous result holds for the Navier problem.
Note that for this last statement, unlike in the previous results concerning the stationary problems,
we have allowed the datum $h$ to depend on both space and time coordinates. The proof makes use of Banach
fixed point theorem and of the potential well techniques and can be found in~\cite{n1}.

We have proven more results concerning the parabolic problem~\eqref{parabolicpde} but we are not going to expose them because of their
more technical nature they require more preparatory results. They concern the asymptotic properties of the solutions to the Dirichlet problem,
that converge to zero in the long time provided $\lambda=0$ and the initial datum is small enough. Also, the blow-up in finite time that
takes place for both Navier and Dirichlet problems in case the initial datum is large enough. These results can be found in~\cite{n1}.
Precisely these results, and in particular those regarding the blow-up, motivate in part the study of self-similar solutions to the problems at hand.
In the following section we outline some facts about this sort of solutions.

\section{The search for self-similar solutions}\label{sss}

In this section we describe an open problem related to the existence of self-similar solutions to partial differential equation~\eqref{parabolicpde}.
For the first time we are going to consider this equation set on all the plane instead of on a bounded domain $\Omega$.
Our first step is setting $\lambda=0$. Next we look for solutions with the following form
$u(r,t) = \frac{1}{t^\beta} \, f\left(\frac{r}{t^\alpha}\right)$,
where $\alpha$ and $\beta$ are two real parameters to be fitted in order to find a closed ordinary differential equation involving the self-similar variable
$\eta=\frac{r}{t^\alpha}$ and the solution to equation~\eqref{parabolicpde} expressed as a function of this variable only $f(\eta)$.

Substituting our self-similar ansatz in~\eqref{parabolicpde} we find that this equation adopts the closed form
\begin{equation}\nonumber
4 f'(\eta)-\eta^4 f'(\eta)-4 \eta f''(\eta)-4 \eta^2 f'(\eta) f''(\eta) + 8 \eta^2 f'''(\eta) + 4 \eta^3 f''''(\eta)=0,
\end{equation}
only if $\alpha=1/4$ and $\beta=0$. Note the simplicity of this fact possibly makes it one of the simplest ways of looking
for self-similar solutions.
Now we have to provide suitable boundary conditions for this ordinary differential equation. Note that by its very nature
the self-similar variable describes a rotationally invariant solution. Therefore we assume the following symmetry conditions on the
solution to equation~\eqref{parabolicpde}: $\partial_r u(0,0,t)=0$ and $\partial_r \Delta_r u(0,0,t)=0$,
for a rotationally invariant $u$, where $\Delta_r(\cdot)=\frac{1}{r}\partial_r[r \partial_r(\cdot)]$.
Finally, we impose that both $u$ and $\Delta u$ decay
to zero as $r \to \infty$. So this leads to the boundary value problem:
\begin{eqnarray}
4 f'(\eta)-\eta^4 f'(\eta)-4 \eta f''(\eta)-4 \eta^2 f'(\eta) f''(\eta) + 8 \eta^2 f'''(\eta) + 4 \eta^3 f''''(\eta)=0,
\nonumber \\
\label{selfsimilar1}
f'(0)=f'''(0)=0, \quad
f(\eta), \, f''(\eta) \to 0 \quad \mbox{as } \eta \to \infty.
\end{eqnarray}
The existence of solutions to this problem automatically implies the existence of solutions of the form
$u(x,y,t) = f\left(\frac{r}{t^{1/4}}\right)$,
to partial differential equation~\eqref{parabolicpde}. The obvious fact that $u \equiv 0$ solves~\eqref{parabolicpde} with $\lambda=0$
and with the assumed boundary conditions directly translates into the
fact that $f \equiv 0$ solves boundary value problem~\eqref{selfsimilar1}.

Noting that equation~\eqref{selfsimilar1} does not depend on $f$ but on its derivatives we can obtain a third order ordinary differential
equation for $g(\eta)=f'(\eta)$.
The corresponding boundary value problem reads
\begin{eqnarray}\nonumber
4 g-\eta^4 g -4 \eta g' -4 \eta^2 g g' + 8 \eta^2 g'' + 4 \eta^3 g'''=0, \\ \label{selfsimilar2}
g(0)=g''(0)=0, \quad
g'(\eta) \to 0 \quad \mbox{as } \eta \to \infty.
\end{eqnarray}
Note that this problem is strongly singular and it is to be solved for $\eta \in [0,\infty)$. Obviously $g \equiv 0$ solves this boundary value problem.
Therefore we will be interested in nontrivial solutions. Since one expects the solutions of (\ref{selfsimilar1}), if any, to behave like a Gaussian,
one should first try to exclude the cases where $f''(0)>0$. We do so in the next statement.

\begin{proposition}\label{prop}
A local solution of the problem
\begin{eqnarray}\label{under}
4 g-\eta^4 g -4 \eta g' -4 \eta^2 g g' + 8 \eta^2 g'' + 4 \eta^3 g'''=0, \\ \nonumber
g(0)=g''(0)=0,\ g'(0)>0,
\end{eqnarray}
blows up in finite time. More precisely, there exists $\overline{\eta}>0$ such that $g'(\eta)>0$ for $\eta\in(0,\overline{\eta})$ and
$\lim_{\eta\to\overline{\eta}}g(\eta)=+\infty$.
\end{proposition}
\begin{proof} Let $g$ be a local solution of (\ref{under}). For contradiction, assume that there exists a first $\eta_0>0$ such that $g'(\eta_0)=0$.
Then $g(\eta)>0$ and $g'(\eta)>0$ for all $\eta\in(0,\eta_0)$. In turn, if we rewrite the equation appearing in~\eqref{under} as
$$
4(\eta^3 g'' - \eta^2 g' + \eta g)'=\eta^2 g (4g'+\eta^2),
$$
this readily shows that $(\eta^3 g'' - \eta^2 g' + \eta g)'>0$ in $(0,\eta_0]$. Since the term inside the bracket vanishes at $\eta=0$ this yields
$\eta^3 g''(\eta) - \eta^2 g'(\eta) + \eta g(\eta)>0$ in $(0,\eta_0]$, which we may also rewrite as
$$g''(\eta) - \frac{g'(\eta)}{\eta} + \frac{g(\eta)}{\eta^2}>0\qquad\mbox{for }\eta\in(0,\eta_0].$$
This proves that $(g'-\frac{g}{\eta})'>0$ and since the term inside this bracket vanishes as $\eta\to0$, this also gives
$g'(\eta_0)>\frac{g(\eta_0)}{\eta_0}>0$, contradicting the characterization of $\eta_0$. We have so proved that
\begin{equation}\label{positivederiviative}
g'(\eta)>0\qquad\mbox{for all }\eta\in[0,\overline{\eta})
\end{equation}
where $g$ is a local solution to (\ref{under}) and $\overline{\eta}$ is the endpoint of its interval of continuation, possibly infinite.
Moreover, what we have seen also proves that
\begin{equation}\label{infinite}
\lim_{\eta\to\overline{\eta}}g(\eta)=+\infty
\end{equation}
in both the cases $\overline{\eta}<\infty$ (blow up in finite time) and $\overline{\eta}=+\infty$ (global solution).\par
Now put $\eta=e^r$, $g(\eta)=h(\log \eta)$, $h(r)=g(e^r)$. Then the equation in (\ref{under}) reads
\begin{equation}\label{hr}
4\left(h'''(r)-h''(r)-h'(r)\right)=4e^rh'(r)h(r)+(e^{4r}-4)h(r),\qquad r\in(-\infty,+\infty),
\end{equation}
while conditions (\ref{positivederiviative}) and (\ref{infinite}) become
\begin{equation}\label{conditions}
h'(r)>0\quad\mbox{for all }r<\overline{r},\qquad \lim_{r\to\overline{r}}h(r)=+\infty,\qquad \overline{r}=\log\overline{\eta}.
\end{equation}
For contradiction, assume that $\overline{r}=+\infty$. From (\ref{hr}) and (\ref{conditions}) we infer that
$$\left(h''(r)-h'(r)-h(r)\right)' \ge e^rh'(r)h(r)\ge 2h'(r)h(r)\quad\mbox{for all }r\ge\log 2.$$
By integrating this inequality over $[\log 2,r]$ we get
$$
h''(r)-h'(r)-h(r)\ge h(r)^2+\gamma\quad\mbox{for all }r\ge\log 2
$$
where $\gamma=h''(\log 2)-h'(\log 2)-h(\log 2)-h(\log 2)^2$. By (\ref{conditions}) we may multiply this inequality by $h'(r)$ and maintain its sense:
$$
h''(r)h'(r)\ge h'(r)^2+h(r)h'(r)+h(r)^2h'(r)+\gamma h'(r)\ge h(r)h'(r)+h(r)^2h'(r)+\gamma h'(r)
$$
for all $r\ge\log 2$.
Let us now integrate this inequality over $[\log 2,r]$; we obtain
$$\frac{1}{2} h'(r)^2\ge \frac{1}{2} h(r)^2+\frac{1}{3} h(r)^3+\gamma h(r)+\delta$$
where $\delta=\frac{1}{2} h'(\log 2)^2-\frac{1}{2} h(\log 2)^2-\frac{1}{3} h(\log 2)^3-\gamma h(\log 2)$. By (\ref{conditions}) we know that there exists $R>\log 2$
such that the latter inequality yields
$h'(r)^2\ge \frac{1}{4} h(r)^3$ for all $r\ge R$.
By taking the square root of this inequality we obtain $\frac{h'(r)}{h(r)^{3/2}}\ge\frac{1}{2}$ which, upon integration over $[R,r]$ yields
$$\frac{1}{h(R)^{1/2}}\ge\frac{r-R}{4}+\frac{1}{h(r)^{1/2}}\qquad\mbox{for all }r\ge R.$$
By letting $r\to+\infty$ we reach a contradiction.\end{proof}

\begin{figure}
\centering \subfigure[]{
\includegraphics[width=0.45\textwidth]{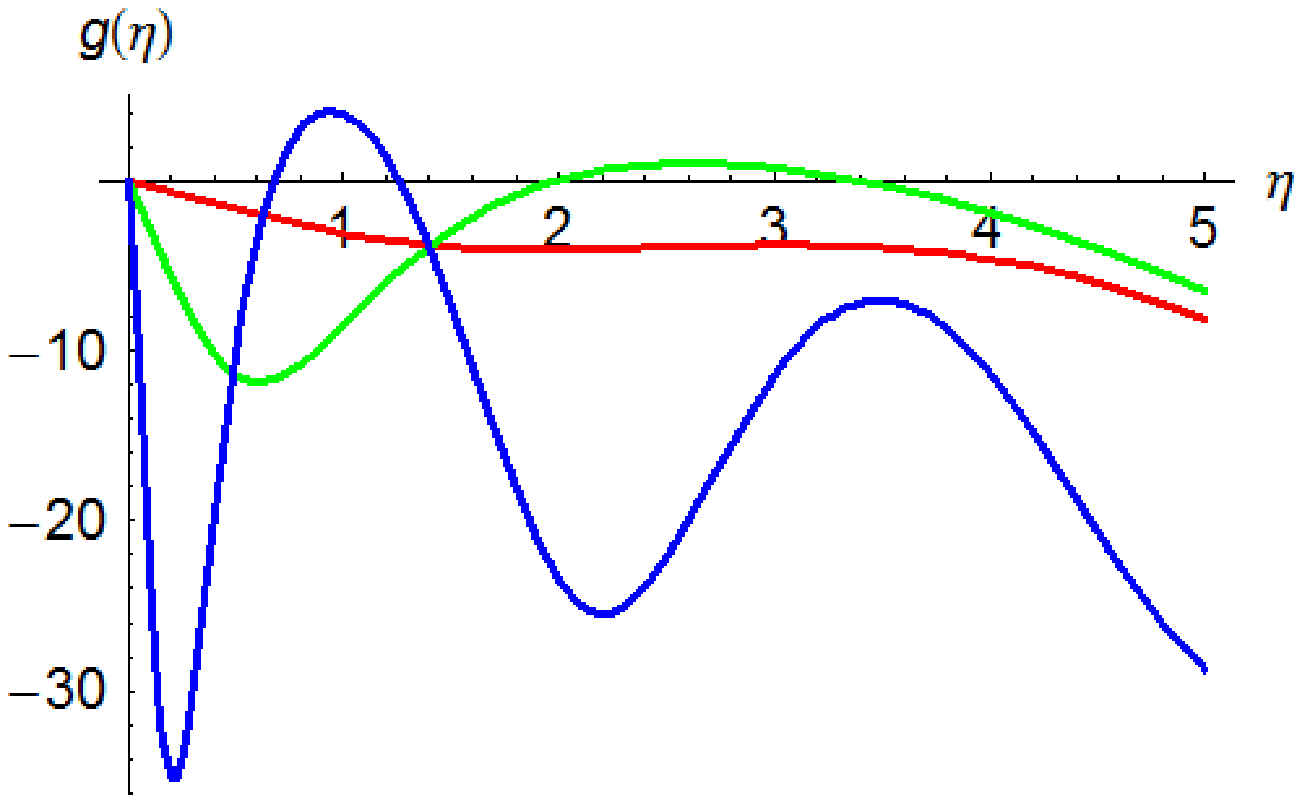}
\label{shot1}} \subfigure[]{
\includegraphics[width=0.45\textwidth]{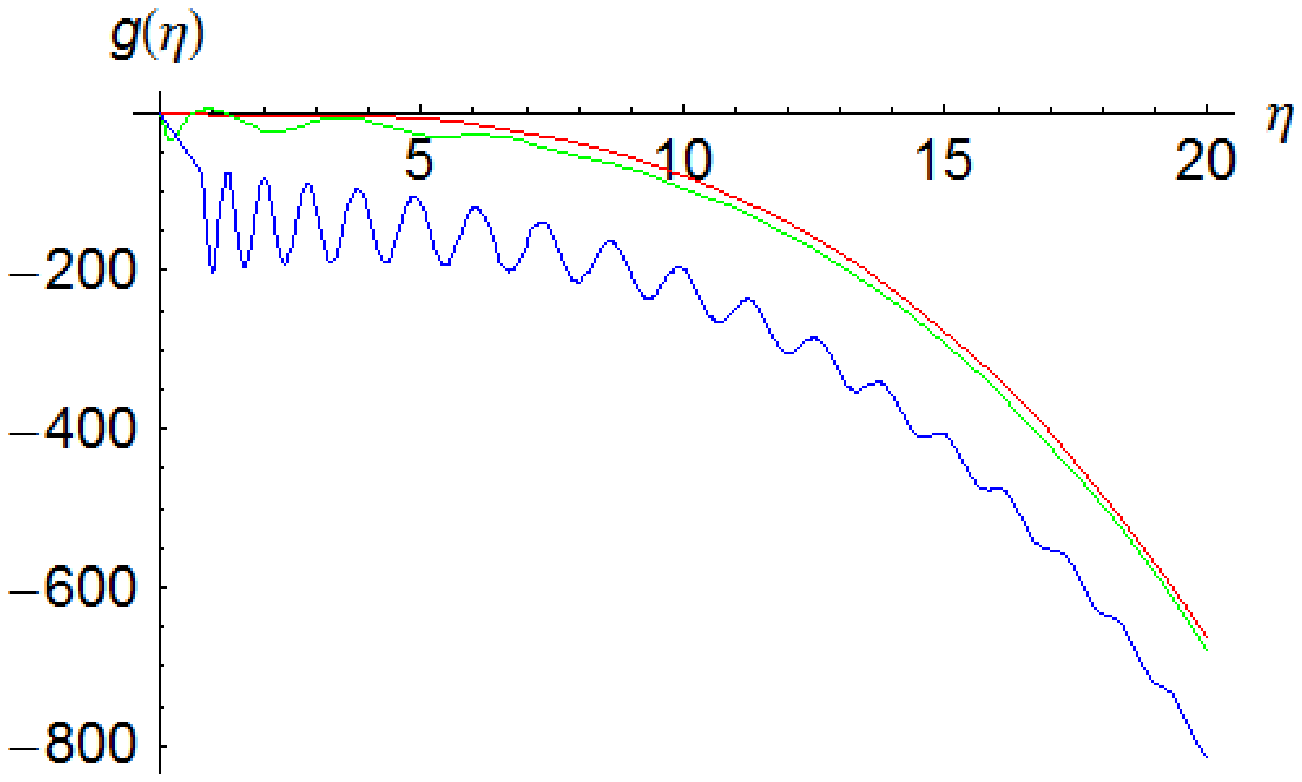}
\label{shot2}} \\ \subfigure[]{
\includegraphics[width=0.45\textwidth]{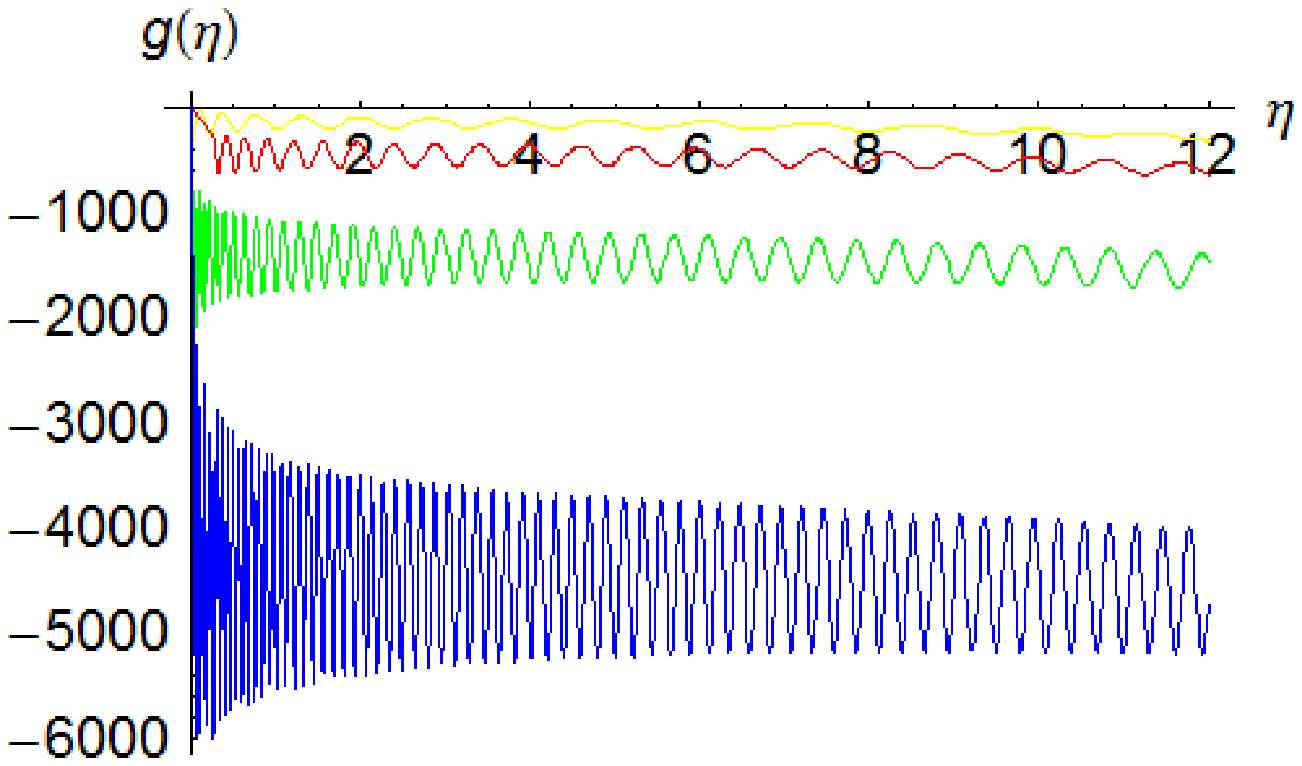}
\label{shot3}} \subfigure[]{
\includegraphics[width=0.45\textwidth]{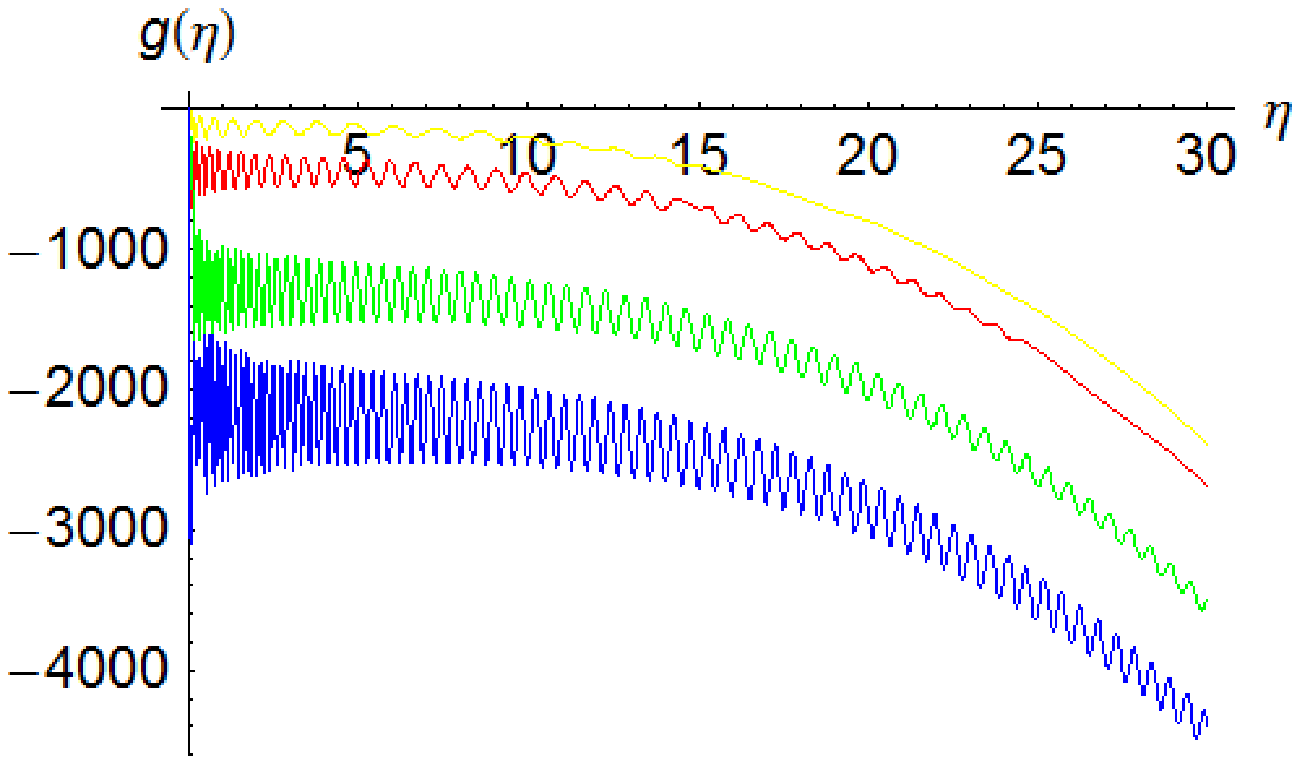}
\label{shot4} } \caption{Numerical solutions of the ordinary differential equation corresponding to boundary value
problem~\eqref{selfsimilar2}. The boundary conditions $g(0)=g''(0)=0$ were employed in this numerical integration. The
third condition is arbitrary and chosen to be: $g'(0)=-1$ (red line), $g'(0)=-10$ (green line) and $g'(0)=-10^2$ (blue line)
in figure~\ref{shot1}; $g'(0)=-1$ (red line), $g'(0)=-10^2$ (green line) and $g'(0)=-10^4$ (blue line)
in figure~\ref{shot2}; $g'(0)=-10^4$ (yellow line), $g'(0)=-10^5$ (red line), $g'(0)=-10^6$ (green line)
and $g'(0)=-10^7$ (blue line) in figure~\ref{shot3}; $g'(0)=-10^4$ (yellow line), $g'(0)=-10^5$ (red line),
$g'(0)=-8.5 \times 10^5$ (green line) and $g'(0)=-2.4 \times 10^6$ (blue line)
in figure~\ref{shot4}.} \label{shots}
\end{figure}

In the case $g'(0)<0$, we have numerically integrated this differential equation, for an example see figure~\ref{shots}.
Based on Proposition \ref{prop} and on these preliminary numerical experiments we conjecture the following result:

\begin{Con}\label{C1}
   There exist no nontrivial solutions to boundary value problem~\eqref{selfsimilar2}.
\end{Con}

It is interesting to highlight one property of the solutions to~\eqref{under} for any $g'(0)$
used in the numerical integrations shown in figure~\ref{shots}. If we assume that $g \in C^3[0,\infty)$, we have
$$
g(\eta) = g'(0) \eta + g'''(0)\frac{\eta^3}{6} + \mathrm{o}(\eta^3), \quad
g'(\eta) = g'(0) + g'''(0)\frac{\eta^2}{2} + \mathrm{o}(\eta^2),
$$
$$
g''(\eta) = g'''(0) \eta + \mathrm{o}(\eta), \quad
g'''(\eta) = g'''(0) + \mathrm{o}(1) \quad \mbox{as } \eta \to 0.
$$
Substituting this into the equation and simplifying we arrive at
$$
\left[ \frac{8}{3}g'''(0)-g'(0)^2 \right]\eta^3 + \mathrm{o}(\eta^3)=0,
$$
that implies $g'''(0)=3 g'(0)^2/8>0$. Consequently $g''(\eta)>0$ in a right neighborhood of $\eta=0$ and thus
$g(\eta)$ is convex in a right neighborhood of $\eta=0$. Note that this feature can be observed in the numerical solutions
in figure~\ref{shot1}.

Let us now make some final remarks. Note first that a proof of conjecture~\ref{C1} would not imply the nonexistence of (any kind of) self-similar
solutions to partial differential equation~\eqref{parabolicpde}.
The existence of these is actually a question that remains open.
Also, the trivial solution $g \equiv 0$ corresponds to a constant $f$,
but the only constant $f$ that solves boundary value problem~\eqref{selfsimilar1} is $f \equiv 0$ due to the extra boundary condition.
As we have already mentioned, this corresponds to the solution $u \equiv 0$ of the partial differential equation.
In connection to this, any nontrivial solution to problem~\eqref{selfsimilar2} gives rise to a nontrivial solution $f$
to~\eqref{selfsimilar1} only if it satisfies an extra cancelation condition.

{\small
{\em Authors' addresses}:
{\em Carlos Escudero}, Universidad Aut\'onoma de Madrid \& ICMAT (CSIC-UAM-UC3M-UCM), Madrid, Spain.
e-mail: \texttt{carlos.escudero@\allowbreak uam.es}.
{\em Filippo Gazzola}, Politecnico di Milano, Milan, Italy.
e-mail: \texttt{filippo.gazzola@\allowbreak polimi.it}.
{\em Robert Hakl}, Institute of Mathematics AS CR, Brno, Czech Republic.
e-mail: \texttt{hakl@\allowbreak ipm.cz}.
{\em Ireneo Peral}, Universidad Aut\'onoma de Madrid, Madrid, Spain.
e-mail: \texttt{ireneo.peral@\allowbreak uam.es}.
{\em Pedro Torres}, Universidad de Granada, Granada, Spain.
e-mail: \texttt{ptorres@\allowbreak ugr.es}.
}

\end{document}